\documentclass[%
 aip,
 amsmath,amssymb,
 reprint,%
jmp]{revtex4-2}

\usepackage{graphicx}% Include figure files
\usepackage{dcolumn}% Align table columns on decimal point
\usepackage{bm}% bold math
\usepackage[mathlines]{lineno}% Enable numbering of text and display math
\usepackage{amsthm}
\usepackage{amsmath}
\usepackage{amssymb}
\usepackage{mathrsfs}
\usepackage{mathtools}
\usepackage{mathptmx}
\usepackage{etoolbox}
\usepackage{xcolor}

\makeatletter
\def\@email#1#2{%
 \endgroup
 \patchcmd{\titleblock@produce}
  {\frontmatter@RRAPformat}
  {\frontmatter@RRAPformat{\produce@RRAP{*#1\href{mailto:#2}{#2}}}\frontmatter@RRAPformat}
  {}{}
}%
\makeatother

\newtheorem{theorem}{Theorem}[section]
\newtheorem*{definition}{Definition}
\newtheorem{corollary}{Corollary}[section]
\newtheorem{proposition}{Proposition}[section]
\newtheorem{lemma}{Lemma}[section]

\begin{document}

\preprint{AIP/123-QED}

\title[Inductive Volume Measure: A Geometric Alternative to Hausdorff Measure]{Inductive Volume Measure: A Geometric Alternative to Hausdorff Measure}
\author{{Luis A.} {Cede\~no-P\'erez}}

\email{luisacp@ciencias.unam.mx}

\affiliation{ Instituto de Ciencias Nucleares, Universidad Nacional Aut\'onoma de M\'exico, AP  70543, Mexico City, Mexico }

\date{\today}

\begin{abstract}
We employ methods of geometry and generalized convergence to construct a geometric measure that serves as an alternative to the integer-dimension Hausdorff measure. This construction prioritizes integration, yields the Area Formula as a byproduct of the construction and the Coarea Formula follows indirectly from the smooth case. Furthermore, this is the smallest measure that satisfies the Area Formula. Though this construction is not as general as Hausdorff measure, it provides a much simpler introduction to the topic and is enough for certain applications to Geometric Analysis.
\end{abstract}



\maketitle

\tableofcontents

\section{Introduction}

\subsection{A motivating problem}

Let $f\colon \mathbb{R}^{2} \to \mathbb{R}$ be integrable in compact subsets of $\mathbb{R}^{n}$, and consider a sequence of smooth curves $\{\gamma_{n}\}_{n\in\mathbb{N}}$ such that $int\;\gamma_{n}$ is precompact and $int\;\gamma_{n+1} \subset int\;\gamma_{n}$. A natural question that arises is if the limit
\begin{equation*}
    \lim_{n\to\infty} \int_{\gamma_{n}}f
\end{equation*}
exists. Furthermore, if $\gamma$ is another closed curve such that $\gamma_{n} \to \gamma$ in an adequate sense does it follow that
\begin{equation*}
    \lim_{n\to\infty} \int_{\gamma_{n}}f = \int_{\gamma}f\,?
\end{equation*}
Questions such as this arise frequently in Analysis and Geometry as well as many other applications (see \cite{Cedeño} for an example in which this problem is relevant). One may be tempted to employ the usual techniques of measure theory, but certain problems arise immediately. For example, it is well known that there exists measures $d\gamma_{n}$ such that
\begin{equation*}
    \int f\;d\gamma_{n} = \int_{\gamma_{n}} f,
\end{equation*}
however, the measure is different for each curve $\gamma_{n}$. To solve this we may consider the sequence of measures $(d\gamma_{n})_{n\in\mathbb{N}}$ and try to establish some convergence in the sense of measures. Once again, this fails since the measures don't even converge setwise.

We may instead regard each $\gamma_{n}$ as a set and find a single measure $H_{1\leq 2}$ such that if $\gamma$ is any smooth curve in $\mathbb{R}^{2}$ then
\begin{equation}\label{EqLenghtFormula}
    \int_{\gamma}f\;dH_{1\leq 2} = \int_{\gamma} f.
\end{equation}
This would allow us to study the limit as the set changes in a way more familiar from measure theory.

In this paper we construct a new measure, which we will call the \textbf{inductive volume measure}, that satisfies generalizations of \eqref{EqLenghtFormula} known as the Area and Coarea Formulas. The most famous measure that satisfies the Area and Coarea Formulas is the Hausdorff measure, the construction of which is rather involved and the proof of its fundamental results even more so (see \cite{Federer} and \cite{Evans}). The inductive volume measure, while not equivalent to the Hausdorff measure, agrees with the Hausdorff measure in a large class of sets and is the smallest Borel measure in $\mathbb{R}^{m}$ that satisfies the Area Formula. Furthermore, the inductive volume measure can be defined just as easily in riemannian manifolds. The loss in generality of the inductive volume measure is countered by a much simpler construction that is perfectly suited for introductions to the topic.

\subsection{Hausdorff Measure}

The Hausdorff measure $H_{m\leq n}$, essentially, measures $m$-dimensional subsets of $n$-dimensional space. Grossly simplifying, Hausdorff measures have 3 uses:
\begin{enumerate}
    \item Study irregular sets (see \cite{Falconer}).
    \item Study non-integer dimension of sets (see \cite{Falconer}).
    \item Study irregular functions in regular sets (manifolds or almost-manifolds) (see \cite{Evans} and \cite{Krantz}).
\end{enumerate}
The first two tasks can be done as soon as the Hausdorff measures are constructed, regardless of integration. For example, \cite{Falconer} studies fractal sets in depth, without developing the integration formulas for Hausdorff measure, namely, the Area and Coarea Formulas. The third one, which is the most related to classical problems of Geometry and Analysis, can not be done until the Area and Coarea Formulas are developed (see \cite{Krantz} or \cite{Evans}). This implies a long learning process that can be overwhelming for those who begin their study of Geometric Measure Theory and other similar subjects.

Let's recall briefly the construction of Hausdorff measures, following \cite{Krantz}. Let $m\in\mathbb{N}$ and $\Omega_{m}$ be the $m$-dimensional volume of the unit ball centered in zero, that is,
\begin{align*}
    \Omega_{m} &= \lambda^{m}(B_{1}(0))\\
    &= \frac{2\pi^{\frac{m}{2}}}{m\Gamma(\frac{m}{2})}.
\end{align*}
For $\delta > 0$ define
\begin{equation*}
    H_{m,\delta}(A) = \inf\left\{\sum_{k=1}^{\infty}\Omega_{m}2^{-k}(diam(B_{k}))^{k}\;\bigg|\;A\subset \bigcup_{k\in\mathbb{N}}B_{k} \textrm{ and }diam(B_{k}) \leq \delta\right\}.
\end{equation*}
Due to Caratheodory's construction (see \cite{Federer} or \cite{Krantz}) it is easy to see that $H_{m,\delta}$ is an outer measure. Futhermore, one can verify that $\delta \longmapsto H_{m,\delta}(A)$ is decreasing, hence the limit
\begin{equation*}
    H_{m}(A) = \lim_{\delta\to 0}H_{m,\delta}(A)
\end{equation*}
exists, although it may be $\infty$. One then checks that $H_{m}$ is still an outer measure due to the monotonicity of $\delta \longmapsto H_{m,\delta}(A)$. One can then apply the usual theory of Caratheodory to restrict $H_{m}$ to a measure on certain $\sigma$-algebra, although is not clear on whether the Borel $\sigma$-algebra $Bo(\mathbb{R}^{n})$ is contained in it. To this end, one shows that $H_{m}$ is a metric outer measure, which means that if $A,B\subset\mathbb{R}^{n}$ are such that $d(A,B) > 0$ then
\begin{equation*}
    H_{m}(A\cup B) = H_{m}(A) + H_{m}(B).
\end{equation*}
The next step is to show that for any metric outer measure the Borel measurable sets are Caratheodory measurable. This implies that $H_{m}$ is defined in a $\sigma$-algebra that contains the Borel $\sigma$-algebra of $\mathbb{R}^{n}$.

All the previous steps are rather technical and complicated, and all this effort is only to show that the Hausdorff measure $H_{m}$ is a properly defined measure in a $\sigma$-algebra that contains the Borel sets. It remains to show how the integral with respect to $H_{m}$ is computed, which is the content of the Area and Coarea Formulas, the proof of which is even more difficult (see \cite{Krantz} and \cite{Evans}).

The advantage of the previous construction is that it is not necessary for $m$ to be a natural, but any positive real number $s$ (the normalization constants $\Omega_{m}2^{-m}$ are simply set to $1$). This allows one to define notions such as \textbf{Hausdorff dimension}. The relevant result is that for any measurable set $A$, if $H_{p}(A) < \infty$ then $H_{q}(A) = 0$ for any $q > p$, and if $H_{p}(A) > 0$ then $H_{q}(A) = \infty$ for any $q < p$. Hence the numbers
\begin{equation*}
    \inf \{p \geq 0\;|\;H_{p}(A) = 0\} \quad\text{and}\quad \sup \{p \geq 0\;|\;H_{p}(A) = \infty\}
\end{equation*}
are both finite and equal. Another advantage is that the domain of $H_{m}$ is much larger than the Borel $\sigma$-algebra, hence, one may study sets that are not Borel measurable. Other references for the topic are \cite{Federer}, \cite{Folland} and \cite{Evans}.

\subsection{Summary of the Construction}

The main ingredient is shifting our attention from the definition of the measure to the definition of the integral, thus the goal is to generate a measure that satisfies the Area and Coarea Formulas. The only requirements for the construction of the inductive volume measure are
\begin{enumerate}
    \item Differential forms.
    \item Generalized convergence, in the form of nets.
\end{enumerate}
Knowledge of differential forms is required in advanced topics of Geometric Measure Theory, as in the theory of currents (see \cite{Federer} and \cite{Krantz}), and generalized convergence is an invaluable tool in almost every area of mathematics that is vaguely related to Topology or Analysis (see \cite{Megginson} and \cite{Macheras}). Hence, knowledge of this preliminaries can be considered advantageous to those who start their journey in Geometric Measure Theory, anyway.

Let $\Omega\subset\mathbb{R}^{n}$ be an open set, $M\subset \Omega$ an $m$-dimensional submanifold and $dV_{M}$ its volume form. The volume form allows us to define the integral of continuous functions in $M$ as the integral of the differential form $f\;dV_{M}$, that is
\begin{equation*}
    \int_{M} f = \int fdV_{M}.
\end{equation*}
It turns out that this defines an element of $C_{c}(M)^{\ast}$ and the Riesz-Markov-Kakutani Theorem provides the existence of a Borel measure $\mu_{M}$ such that integration with respect to this measure is the functional, that is,
\begin{equation*}
    \int f\;d\mu_{M} = \int_{M} f.
\end{equation*}
The measure $\mu_{M}$ can be lifted to a measure on $\Omega$ by the formula
\begin{equation*}
    \tilde{\mu}_{M}(A) = \mu_{M}(A\cap M),
\end{equation*}
however, it only actually measures subsets of $M$. 

Thus, if $\mathcal{M}_{m}$ is the class of $m$-dimensional submanifolds of $\mathbb{R}^{n}$ then $(\mu_{M})_{M\in\mathcal{M}_{m}}$ resembles an increasing net of measures. This fails because the class $\mathcal{M}_{m}$ is not a directed set under inclusion, since the union of manifolds is not a manifold. The strategy is improving the class $\mathcal{M}_{m}$ so that the resulting class is a directed set. This will be done by considering finite unions of Borel subsets contained in $m$-dimensional submanifolds. The resulting net will actually be an inductive system of measures and the inductive volume measure will be its limit. The Area Formula will follow immediately from properties of inductive systems of measures. The Coarea Formula will not be as immediate, but will follow from an extension argument from the smooth case. Furthermore, we will present a new proof for the fact that every measure that satisfies the Area Formula also satisfies the Coarea Formula. We will construct the desired measure on $\mathbb{R}^{n}$ but the present results carry to the setting of riemannian manifolds with scarcely any modification, which is another advantage of this construction.

\section{Preliminaries}

In this section we collect the basic results needed for our construction. Proofs will be omitted to simplify the exposition, but can be found in the references given at the end of each subsection.

\subsection{Generalized Convergence}

A preorder in a set $I$ is a relation $\preceq$ such that
\begin{enumerate}
    \item $i\preceq i$ for any $i\in I$.
    \item If $i\preceq j$ and $j\preceq k$ then $i\preceq k$.
\end{enumerate}

It is not necessary that any two elements in $I$ can be compared. For example, if $X$ is any set and $I = \mathcal{P}(X)$ then the relation $A \preceq B$ if $A\subset B$ is a preorder in $I$ known as \textbf{direct inclusion}. Furthermore, given $A,B\in I$ it may happen that neither $A\subset B$ nor $B\subset A$. Another common preorder is given by $A\preceq B$ if $A\supset B$, known as \textbf{inverse inclusion}.

\begin{definition}
A preordered set $(I,\preceq)$ is a \textbf{directed set} if for any $i,j\in I$ there exists $k\in I$ such that $i\preceq k$ and $j\preceq k$.
\end{definition}

There are a couple of very familiar examples of directed sets.
\begin{enumerate}
    \item If $I = \mathcal{P}(X)$ then $I$ is a directed set with either direct or inverse inclusion.
    \item Both $\mathbb{N}$ and $\mathbb{R}$ are directed sets with their usual orders.
\end{enumerate}

\begin{definition}
Given a set $X$, a net is a function $s\colon I \to X$ such that $I$ is a directed set. The net is denoted as $(s_{i})_{i\in I}$ and each of its elements is denoted as $s_{i}$ as a parallel to sequences.
\end{definition}

Nets were designed to capture the properties of topological spaces in a similar way that sequences do in metric spaces. Most importantly, any topological space defines a criterion for convergence of nets.

\begin{definition}
Let $(X,\tau)$ be a topological space. We say that a net $(x_{i})_{i\in I}$ in $X$ converges to $x\in X$ if for every $U$ neighbourhood of $x$ there exists $i_{U}\in I$ such that $x_{i}\in U$ whenever $i \succeq i_{U}$.
\end{definition}

It is clear that every sequence is a net and that the convergence of sequences is a particular case of the convergence of nets. Since we will work with measures we will work in the extended real numbers, which we will denote as $[-\infty,\infty]$. Since this is a topological space, the previous definition is valid. The relevant result is the following.

\begin{proposition}
Every increasing net $(x_{i})_{i\in I}$ in $[-\infty,\infty]$ is convergent in the usual topology of $[-\infty,\infty]$.
\end{proposition}

Generalized Convergence is commonly employed in Topology and Functional Analysis. Further references for the topic are \cite{Megginson} and \cite{Dugundji}.

\subsection{Generalized Limits of Measures}

Generalized convergence is one of the most powerful tools in modern mathematics. This theory can be applied to the context of Measure Theory to obtain powerful results regarding the so-called inductive systems of measures. We recall the most important results.

\begin{theorem}
Let $(\Omega,\Sigma)$ be a measurable space and $(\mu_{i})_{i\in I}$ a net of measures such that for each $E\in\Sigma$ the net $(\mu_{i}(E))_{i\in I}$ is increasing. If $\nu\colon\Sigma \to [0,\infty]$ is the function given by
\begin{equation*}
    \nu(E) = \lim_{i} \mu_{i}(E)
\end{equation*}
then $\nu$ is a measure.
\end{theorem}

In the preceding situation we say that $(\mu_{i})_{i\in I}$ is an increasing net of measures. The function $\nu$ is known as the \textbf{generalized limit} of the net $(\mu_{i})_{i\in I}$ and is denoted by
\begin{equation*}
    \nu = \lim_{i}\mu_{i}.
\end{equation*}

\begin{theorem}
Let $(\Omega,\Sigma)$ be a measurable space, $(\mu_{i})_{i\in I}$ an increasing net of measures and $\nu = \lim_{i}\mu_{i}$. If $f\in M^{+}(\Omega,\Sigma)$ then
\begin{equation*}
    \int f\;d\nu = \lim_{i}\int f\;d\mu_{i}.
\end{equation*}
Furthermore, $f\in M(\Omega,\Sigma)$ is integrable if and only if $f^{+}$ are $f^{-}$ are integrable and the previous formula is valid in this case.
\end{theorem}

The previous results can be strengthened if $\mu_{i}$ restricts to $\mu_{j}$ whenever $j \preceq i$.

\begin{definition}
Let $(\Omega,\Sigma)$ be a measurable space and $(\Omega_{i})_{i\in I}$ an increasing net in $\Sigma$ such that $\Omega_{i}\nearrow\Omega$. An \textbf{inductive system of measures} is an increasing net of measures in $(\mu_{i})_{i\in I}$ in $\Sigma$ such that if $i \succeq j$ then for each $A\in\Sigma$ the equation
\begin{equation*}
    \mu_{i}(A\cap\Omega_{j}) = \mu_{j}(A\cap\Omega_{j})
\end{equation*}
is verified. The previous condition is known as the \textbf{compatibility condition}.
\end{definition}

Intuitively, an inductive system of measures is such that as $i$ increases the value of the measures is fixed in $\Omega_{i}$ as $\mu_{i}$. Thus, the inductive system extends the previous measures as the index increases.

Define $\Sigma_{i} = \Sigma\restriction_{\Omega_{i}} = \{A\cap\Omega_{i}\;|\;A\in\Sigma\}$. Note that it suffices that each $\mu_{i}$ is defined in $\Sigma_{i}$, since we can replace it by
\begin{equation*}
    \tilde{\mu}_{i}(A) = \mu_{i}(A\cap\Omega_{i}).
\end{equation*}
In this case the measures $\tilde{\mu_{i}}$ are defined in $\Sigma$ and as such can replace the original net.

\begin{definition}
If $(\mu_{i})_{i\in I}$ is an inductive system of measures we define its \textbf{inductive limit} $\mu$ as the measure
\begin{equation*}
    \mu = \lim_{i}\mu_{i}.
\end{equation*}
\end{definition}

It follows from the previous theorem that if $f\in M^{+}(\Omega,\Sigma)$ then
\begin{equation*}
    \int f\;d\mu = \lim_{i}\int f\;d\mu_{i}.
\end{equation*}
The additional property obtained from the compatibility condition is the following.

\begin{theorem}\label{TeoSistemaInductivoMedidas}
Let $(\mu_{i})_{i\in I}$ be an inductive system of measures and $\mu$ its inductive limit. If $f\in M^{+}(\Omega,\Sigma)$ vanishes outside of $\Omega_{i_{0}}$ then
\begin{equation*}
    \int f\;d\mu = \int f\;d\mu_{i_{0}}.
\end{equation*}
\end{theorem}

More on inductive systems of measures can be found in \cite{Macheras}.

\subsection{Geometry of Volumes}

Before applying the theory of inductive systems of measures we need to recall the specifics of volume measurement developed in Differential Geometry. Let's begin with the integration of differential forms. Given a continuous $n$-form on $\mathbb{R}^{n}$ it must be of the form
\begin{equation*}
    \omega = f \;dx^{1}\wedge\cdots\wedge dx^{n},
\end{equation*}
hence we define
\begin{equation*}
    \int f \;dx^{1}\wedge\cdots\wedge dx^{n} = \int f \;dx^{1}\cdots dx^{n}.
\end{equation*}
If $M$ is an $n$-dimensional manifold and $\omega$ is a continuous $n$-form with support covered by a coordinate patch $\varphi\colon U \subset \mathbb{R}^{n} \to M$ that preserves orientation then $\varphi^{\ast}\omega$ is a continuous $n$-form in $\mathbb{R}^{n}$, hence we define
\begin{equation*}
    \int_{M} \omega = \int_{U} \varphi^{\ast}\omega.
\end{equation*}
Finally, if $\omega$ is a continuous $n$-form and $(U_{i})_{i\in I}$ is a cover of its support by coordinate patches then consider a partition of unity $(\varphi_{i})_{i\in I}$ subordinated to the cover, then $\varphi_{i}\omega$ is continuous $n$-form with support in a coordinate patch, hence can be integrated as we previously discussed. We define
\begin{equation*}
    \int_{M}\omega = \sum_{i}\int_{M}\varphi_{i}\omega.
\end{equation*}
This does not depend on the cover or the partition of unity used. The relation of this integral with Measure Theory is given by the following result.

\begin{theorem}
Let $\omega$ be an $n$-form in an orientable manifold $M$ of dimension $n$. There is a unique measure $\mu_{\omega}$ in $Bo(\tau_{M})$ such that
\begin{equation*}
    \int f\;d\mu_{\omega} = \int_{M} f \omega
\end{equation*}
 for any $f\in C_{c}(M)$.
\end{theorem}

This follows from the Riesz-Markov-Kakutani Theorem applied to compact subsets $K$ of $M$, on which the functional $f\longmapsto \int_{K} f\omega$ is continuous and hence generated by a measure.

The integral of an arbitrary differential form is not necessarily related to volumes, mainly because the integral may not be a positive functional, that is, it may not happen that
\begin{equation*}
    \int_{M} f\omega \geq 0
\end{equation*}
whenever $f \geq 0$. Hence, the measure $\mu_{\omega}$ is in general a signed measure. When the manofld is riemannian, however, there exists a distinguished $n$-form such that integration is a positive functional.

\begin{definition}
Let $(M,g)$ be an orientable riemmanian manifold. The unique $n$-form $V_{g}$ such that for any orthonormal frame $\{E_{i}\}_{i=1}^{n}$ (orthonormal with respect to $g_{p}$)
\begin{equation*}
    V_{g}(E_{1},\ldots,E_{n}) = 1
\end{equation*}
is known as the \textbf{volume form} of $(M,g)$.
\end{definition}

The volume form exists and is unique whenever $(M,g)$ is orientable. If coordinates are given and the metric tensor has coordinate components $g_{ij}$ then
\begin{equation*}
    V_{g} = \sqrt{det\;g_{ij}}\;dx^{1}\wedge \cdots \wedge dx^{n}.
\end{equation*}

Since the volume form in a riemannian manifold is a distinguished form we can define the integral of continuous functions using this form. Precisely, if $f\in C_{c}(M)$ we define
\begin{equation*}
    \int_{M} f = \int fV_{g}.
\end{equation*}
This functional turns out to be positive in the sense that if $f\geq 0$ then
\begin{equation*}
    \int_{M} f \geq 0.
\end{equation*}
Thus, the corresponding measure is a positive measure by the positive version of the Riesz-Markov-Kakutani Theorem.

\begin{corollary}
Let $M$ be an orientable riemannian manifold of dimension $n$. There is a unique positive measure $\mu_{M}$ in $Bo(\tau_{M})$ such that 
\begin{equation*}
    \int f\;d\mu_{M} = \int_{M} f
\end{equation*}
for any $f\in C(M)$.
\end{corollary}

This measure is precisely what we will need for our purposes.

\begin{definition}
Let $M$ be an orientable riemannian manifold of dimension $n$. The measure $\mu_{M}$ is known as the \textbf{volume measure} of $M$.
\end{definition}

Since we will work with submanifolds of $\mathbb{R}^{n}$ we can always assume that the volume form exists and is the pullback of the volume form of $\mathbb{R}^{n}$ under inclusion.

The previous results can be found in standard Differential Topology and Geoemetry textbooks, such as \cite{Lee} and \cite{Tu}.

\section{Inductive Volume Measure}

We can now construct the inductive volume measure. The basic idea is to use the measures $\mu_{M}$ where $M$ is an $m$-dimensional submanifold of $\mathbb{R}^{n}$ to construct an inductive system of measures.  Theorem \ref{TeoSistemaInductivoMedidas} will then yield the Area Formula for the inductive limit. The Coarea Formula will require some more work, but will follow easily from the smooth case. The results of this section remain valid if $\mathbb{R}^{n}$ is replaced by a riemannian manifold and the proofs are the same. The only noteworthy exception is lemma \ref{LemaIntersección} (see subsection \ref{SecRiemannian}).

\subsection{Construction and Area Formula}

Fix $n,m\in\mathbb{N}$ such that $m\leq n$. To each $m$-dimensional submanifold $M$ of $\mathbb{R}^{n}$ we can assign its volume measure $\mu_{M}$, which suggests how to define a net of measures. However, the collection of $m$-dimensional submanifolds does not form a directed set under inclusion, because the union of manifolds is not in general a manifold. To solve this problem we first establish two fundamental facts.

First, recall that if $\mathcal{C}$ is a class of subsets of $\mathbb{R}^{n}$ and $A$ is another subset then we can define
\begin{equation*}
    \mathcal{C}\restriction_{A} = \{C\cap A\;|\;C\in\mathcal{C}\}.
\end{equation*}
It is easily verified that
\begin{equation*}
    \sigma(\mathcal{C}\restriction_{A}) = \sigma(\mathcal{C})\restriction_{A}.
\end{equation*}
If $S\subset\mathbb{R}^{n}$ then
\begin{equation*}
    \sigma(\tau_{\mathbb{R}^{n}}\restriction_{S}) = \sigma(\tau_{\mathbb{R}^{n}})\restriction_{S}.
\end{equation*}
If $S$ is also a submanifold then $\tau_{S} = \tau_{\mathbb{R}^{n}}\restriction_{S}$, hence
\begin{equation*}
    \sigma(\tau_{S}) = \sigma(\tau_{\mathbb{R}^{n}})\restriction_{S}.
\end{equation*}
Thus if $A$ is Borel measurable subset of $S$ then $A$ is a Borel subset of $\mathbb{R}^{n}$ and if $B$ is a Borel subset of $\mathbb{R}^{n}$ then $B\cap S$ is a Borel subset of $S$. This implies that there is no need to distinguish between both classes of Borel subsets.

The second fundamental fact we need is a little more involved and is best stated and proved as a lemma.

\begin{lemma}\label{LemaIntersección}
For any pair $M$ and $N$ of $m$-dimensional submanifolds their volume measures $\mu_{M}$ and $\mu_{N}$ satisfy
\begin{equation*}
    \mu_{M}(M\cap N) = \mu_{N}(M\cap N).
\end{equation*}
The same relation is valid for any Borel set contained in $M\cap N$.
\end{lemma}
Note that the proof of the lemma is trivially true if we assume the Area Formula for Hausdorff measures. However, for the sake of avoiding this result, we will give an entirely differential-geometric proof.
\begin{proof}
To simplify notation, whenever we write a chart as $\varphi\colon \mathbb{R}^{m} \to \mathbb{R}^{n}$, we will mean that $\varphi$ is defined between suitable neighbourhoods of the domain and codomain.

Due to $\sigma$-additivity it is enough to establish the result in local charts. Given a point in $M\cap N$ select a slice chart $\varphi\colon \mathbb{R}^{n} \to \mathbb{R}^{n}$ for $M$, that is,
the function $\varphi_{M}\colon \mathbb{R}^{m} \to \mathbb{R}^{n}$ given by
\begin{equation*}
    \varphi_{M}(x) = \varphi(x,0)
\end{equation*}
is a chart for $M$, in adequate neighbourhoods. Since $N$ is a submanifold, we can employ the Implicit Function Theorem to obtain a function $F\colon \mathbb{R}^{m} \to \mathbb{R}^{n-m}$ such that the function $\varphi_{N}\colon \mathbb{R}^{m} \to \mathbb{R}^{n}$ given by
\begin{equation*}
    \varphi_{N}(x) = \varphi(x,F(x))
\end{equation*}
is a chart for $N$. (If the last $n-m$ components can't be solved for $F$ we consider $\phi(x,y) = \varphi(x+y,y)$, which is still a slice chart in which the function $F$ can be obtained). Hence the set $M\cap N$ is defined in coordinates as the set where $F(x) = 0$. The differentials of the charts $\varphi_{M}$ and $\varphi_{N}$ are
\begin{equation*}
    D\varphi_{M} = D\varphi
    \left(\begin{array}{c}
         Id \\
          0
    \end{array}\right)
    \quad\text{and}\quad D\varphi_{N} = D\varphi
    \left(\begin{array}{c}
         Id \\
          DF
    \end{array}\right).
\end{equation*}
The metrics given by these charts are
\begin{equation*}
    g^{M} = D\varphi_{M}^{t} D\varphi_{M} \quad\text{and}\quad g^{N} = D\varphi_{N}^{t} D\varphi_{N},
\end{equation*}
hence the densities $\sqrt{det\;g^{M}}$ and $\sqrt{det\;g^{N}}$ are equal in $M\cap N$ precisely at points $x$ where $DF_{x} = 0$. The set $M\cap N$ hence can be divided in the measurable sets where $DF_{x} \neq 0$ and $DF_{x} = 0$. In the first case the submanifolds are tranverse and their intersection is a submanifold of codimension $2(n - m)$ and hence null with respect to both $\mu_{M}$ and $\mu_{N}$. In the second case, both densities coincide and hence the measures $\mu_{M}$ and $\mu_{N}$ have the same value on this set.
\end{proof}

We can now construct the net of measures we will use. Let's denote by $\mathcal{C}_{m}$ the class of Borel subsets that are contained in an $m$-dimensional submanifold of $\mathbb{R}^{n}$. If $C\in\mathcal{C}_{m}$ then there exists an $m$-dimensional submanifold $S$ such that $C\subset S$. We consider the volume measure $\mu_{S}$ defined in the previous section. Since $C\subset S$ is a Borel set we can define
\begin{equation*}
    \mu_{C}(A) = \mu_{S}(A\cap C),
\end{equation*}
which is a Borel measure in $\mathbb{R}^{n}$. Thus, to every set in $\mathcal{C}_{m}$ we can associate a measure. Furthermore, lemma \ref{LemaIntersección} shows that the measure does not depend on the submanifold $S$.

Define $I_{m}$ to be the class of finite disjoint unions of elements in $\mathcal{C}_{m}$. It is clear that for any such collection $\tilde{M} = \bigcup_{i=1}^{k}C_{i}$ we can define
\begin{equation*}
    \mu_{\tilde{M}} = \sum_{i=1}^{k}\mu_{C_{i}}.
\end{equation*}
Note that integration with respect to these measures is still given by the measure generated by volume forms. Once again, lemma \ref{LemaIntersección} shows that the measure does not depend on the collection $\{C_{i}\}_{i=1}^{k}$.

We now show that $I_{m}$ is a directed set. To see this, consider $A,B\in\mathcal{C}_{m}$. By definition, there exists submanifolds $M$ and $N$ such that $A\subset M$ and $B\subset N$. Then
\begin{align*}
    A\cup B = A \uplus (B\setminus (A\cap B)).
\end{align*}
Since
\begin{align*}
    B\setminus (A\cap B) \subset B \subset N
\end{align*}
we have that $B\setminus (A\cap B)\in \mathcal{C}$. Thus, $A\cup B$ can be written as a disjoint unions of elements in $\mathcal{C}_{m}$. Repeating this procedure a finite number of times, we obtain that $I_{m}$ is indeed a directed set.

It follows that
\begin{equation*}
    (\mu_{\tilde{M}})_{\tilde{M}\in I_{m}}
\end{equation*}
is a net of measures in $Bo(\Omega)$.

To establish that the net is actually an inductive system we first make the following observations. If $M$ and $S$ are $m$-dimensional submanifolds such that $M\subset S$ and $\iota_{M\subset S}$ is the inclusion of $M$ in $S$ and $\omega_{M}$ and $\omega_{S}$ are their volume forms then
\begin{equation*}
    \omega_{M} = \iota_{M\subset S}^{\ast}\omega_{S},
\end{equation*}
which implies that
\begin{equation*}
    dV_{S}\restriction_{M} = dV_{M}.
\end{equation*}
Thus, both forms generate the same measures in $M$, that is
\begin{equation*}
    \mu_{M}(A\cap M) = \mu_{S}(A\cap M)
\end{equation*}
for any Borel set $A$. Furthermore, by definition of $\mu_{C}$, the previous equation can be rewritten as
\begin{equation}\label{EqCompatibilidadBorel}
    \mu_{C}(A\cap C) = \mu_{D}(A\cap C)
\end{equation}
whenever $C,D\in\mathcal{C}_{m}$ and $C\subset D$.

We now show that this is an inductive system. Consider $\tilde{M} = \biguplus_{i=1}^{k}C_{i}$ and $\tilde{N}  = \biguplus_{j=1}^{l}D_{j}$ in $I_{m}$ such that $\tilde{M} \subset \tilde{N}$. Manipulating the collections we may assume that $k = l$ and $C_{i} \subset D_{i}$, hence the definitions and equation (\ref{EqCompatibilidadBorel}) imply that
\begin{align*}
    \mu_{\tilde{M}}(A\cap \tilde{M}) &= \sum_{j=1}^{k} \mu_{C_{j}}(A\cap \tilde{M})\\
    &= \sum_{j=1}^{k} \mu_{C_{j}}(A \cap \tilde{M}\cap C_{j})\\
    &= \sum_{j=1}^{k} \mu_{D_{j}}(A \cap \tilde{M} \cap C_{j})\\
    &= \sum_{j=1}^{k} \mu_{D_{j}}(A \cap \tilde{M})\\
    &= \mu_{\tilde{N}}(A\cap \tilde{M}).
\end{align*}
This is the compatibility condition. The fact that $(\mu_{\tilde{M}})_{\tilde{M}\in I_{m}}$ is increasing follows in a similar fashion since the sets are disjoint and the summands are non-negative. This shows that $(\mu_{\tilde{M}})_{\tilde{M}\in I_{m}}$ is actually an inductive system.

\begin{definition}
We define the \textbf{inductive volume measure} of \textbf{local dimension} $m$ and \textbf{ambient dimension} $n$ as the measure $\mathcal{H}_{m\leq n}\colon Bo(\Omega) \to [0,\infty]$ given by
\begin{equation*}
    \mathcal{H}_{m\leq n} = \lim_{\tilde{M}}\mu_{\tilde{M}}.
\end{equation*}
\end{definition}

The notation $\mathcal{H}_{m\leq n}$ is chosen simply to parallel Hausdorff measures. Note that the inductive volume measure requires the local dimension to be an integer, as the tools of Differential Geometry wouldn't be available otherwise, which is a disadvantage with respect to Hausdorff measure. Applying the results of the inductive limit of measures we obtain the following.

\begin{theorem}[Area Formula. Particular Case.]
If $A\in Bo(\Omega)$ is covered by a sequence $(M_{k})_{k\in\mathbb{N}}$ of $m$-dimensional submanifolds then
\begin{equation*}
    \mathcal{H}_{m\leq n}(A) = \sum_{k=1}^{\infty}\mu_{M_{k}}(A\cap M_{k}).
\end{equation*}
If $f$ is an integrable function with respect to $\mathcal{H}_{m\leq n}$ and $\{M_{k}\}_{k\in\mathbb{N}}$ is a disjoint sequence of submanifolds then
\begin{equation*}
    \int_{\bigcup\limits_{k\in\mathbb{N}}M_{k}} f\;d\mathcal{H}_{m\leq n} = \sum_{k=1}^{\infty}\int f\;d\mu_{M_{k}}.
\end{equation*}
Likewise, if $f$ is integrable with respect to $\mu_{M_{k}}$ for each $k\in\mathbb{N}$ and
\begin{equation*}
    \sum_{k=1}^{\infty}\int f^{+}\;d\mu_{M_{k}} < \infty \quad\text{and}\quad \sum_{k=1}^{\infty}\int f^{-}\;d\mu_{M_{k}} < \infty
\end{equation*}
then $f$ is integrable with respect to $\mathcal{H}_{m \leq n}$ and the previous formula is valid.
\end{theorem}

Another disadvantage of the inductive volume measure with respect to Hausdorff measure is that $\mathcal{H}_{m\leq n}$ is only defined on \textit{Borel subsets}, while Hausdorff measure is defined on a much larger $\sigma$-algebra. This is the reason for our construction to be insufficient to study irregular sets. Another disadvantage is that we lack an explicit formula for the Hasudorff measure as we constructed it, since we gave priority to knowing how to compute its integrals. However, by selecting $f$ as a charateristic function in the previous result we obtain the following.

\begin{corollary}
If a Borel measurable set $A$ is contained in the union of a disjoint sequence of $m$-dimensional manifolds $(M_{k})_{n\in\mathbb{N}}$ then
\begin{equation*}
    \mathcal{H}_{m\leq n}(A) = \sum_{k=1}^{\infty}\mu_{M_{k}}(A\cap M_{k}).
\end{equation*}
\end{corollary}

We will develop some more formulas for $\mathcal{H}_{m\leq n}$ in subsection \ref{SubsecComparación}. We insist that the advantage of our method is that it is much simpler than the usual one and immediately yields the Area Formula.

\subsection{Coarea Formula}

The second fundamental result about the inductive volume measure that we will prove is the Coarea Formula. We will actually give a new proof that any Borel measure that satisfies the Area Formula also satisfies the Coarea Formula. Since the inductive volume measure is defined through Differential Geometry, we first establish the smooth version of this result. We follow the proof of \cite{Chavel}.

\begin{theorem}[Smooth Coarea Formula]
Let $H\colon \mathbb{R}^{n} \to \mathbb{R}$ be a smooth function and for every regular value $t$ let $dV_{t}$ be the volume form of $H^{-1}(\{t\})$. If $dV$ is the volume form of $\mathbb{R}^{n}$ then
\begin{equation*}
    dV = \frac{1}{|\nabla H|} dt\wedge dV_{t}
\end{equation*}
for every regular value $t$. In consequence, for any continuous function $f$ we have that
\begin{equation}\label{EqCoarea}
    \int_{H^{-1}([a,b])} f\;dV = \int_{a}^{b}\int_{H^{-1}(\{t\})} \frac{f}{|\nabla H|}\;dV_{t}\;dt.
\end{equation}
\end{theorem}
\begin{proof}
Let $\Phi_{t}$ be the flow of $\nabla H$ in $H^{-1}(\{t\})$. For sufficiently small $t$ and $x\in H^{-1}(\{t\})$ the flow $\Phi_{t}(x)$ defines smooth coordinates for $\mathbb{R}^{n}$. The one form $dt$ is defined as the unique one form such that
\begin{equation*}
    dt(\nabla H) = |\nabla H|^{2}
\end{equation*}
and
\begin{equation*}
    dt(v) = 0
\end{equation*}
if $v\in T_{x}H^{-1}(\{t\})$. If $\{v_{i}\}_{i=1}^{n-1}$ is an orthonormal frame for $T_{x}H^{-1}(\{t\})$ then $\left\{\frac{\nabla H}{|\nabla H|}\right\}\cup \{v_{i}\}_{i=1}^{n-1}$ is an orthonormal frame for $T_{x}\mathbb{R}^{n}$. It follows that
\begin{align*}
    \frac{1}{|\nabla H|} dt\wedge dV_{t} \left(\frac{\nabla H}{|\nabla H|},v_{1},\ldots,v_{n-1}\right) &= \frac{1}{|\nabla H|}dt\left(\frac{\nabla H}{|\nabla H|}\right) dV_{t} (v_{1},\ldots,v_{n-1})\\
    &= \frac{1}{|\nabla H|^{2}}dt(\nabla H)\\
    &= 1,
\end{align*}
hence
\begin{equation}\label{EqCoareaSmooth}
    dV = \frac{1}{|\nabla H|} dt\wedge dV_{t},
\end{equation}
as stated. For the equality of integrals, note that by Sard's Theorem the set of critical values has measure zero, hence the function
\begin{equation*}
    t \longmapsto \int \frac{f}{|\nabla H|}\;dV_{t}
\end{equation*}
is defined except for a measure zero set, hence it can be integrated and equation (\ref{EqCoareaSmooth}) yields (\ref{EqCoarea}).
\end{proof}

By the Area Formula, the last conclusion of the theorem can be rewritten as
\begin{equation}\label{EqCasiCoarea}
    \int_{H^{-1}([a,b])} f\;d\lambda^{n} = \int_{a}^{b}\int_{H^{-1}(\{t\})} \frac{f}{|\nabla H|}\;d\mathscr{H}_{n-1\leq n}\;dt,
\end{equation}
for any continuous function $f$. We wish to strengthen the result to non-continuous functions, which is what is known as the general Coarea Formula. The proof of this fact is not as clean as our proof of the Area Formula, but is still easier than the usual proof. We begin with a simple trick. If we choose $f = 1$ in (\ref{EqCasiCoarea}) we find that
\begin{equation*}
    \lambda^{n}(H^{-1}([a,b])) = \int_{a}^{b}\int_{H^{-1}(\{t\})} \frac{1}{|\nabla H|}\;d\mathscr{H}_{n-1\leq n}\;dt.
\end{equation*}
Since $H$ is any smooth function on a manifold we can consider an open set $A$ and the restriction $H\restriction_{A}\colon A \to \mathbb{R}$ and the previous equation should remain valid. Since $H\restriction_{A}^{-1}([a,b]) = H^{-1}([a,b])\cap A$ and $H\restriction_{A}^{-1}(\{t\}) = H^{-1}(\{t\})\cap A$ we find that
\begin{equation*}
    \lambda^{n}(H^{-1}([a,b])\cap A) = \int_{a}^{b}\int_{H^{-1}(\{t\})\cap A} \frac{1}{|\nabla H|}\;d\mathscr{H}_{n-1\leq n}\;dt
\end{equation*}
for any open set $A$. From now on we consider the total space to be $H^{-1}([a,b])$ to clean notation a bit. From this simple trick we find the following.

\begin{lemma}
The equation
\begin{equation}\label{EqIntegralMeasure}
    \lambda^{n}(A) = \int_{a}^{b}\int_{H^{-1}(\{t\})\cap A} \frac{1}{|\nabla H|}\;d\mathscr{H}_{n-1\leq n}\;dt
\end{equation}
is valid for any open set $A$. In consequence, the function
\begin{equation*}
    t \longmapsto \int_{H^{-1}(\{t\})\cap A} \frac{1}{|\nabla H|}\;d\mathscr{H}_{n-1\leq n}
\end{equation*}
is measurable whenever $A$ is open.
\end{lemma}

Using the usual approximation techniques by simple functions and the Monotone Convergence Theorem, the Coarea Formula will be proven if we can show that equation (\ref{EqIntegralMeasure}) is valid for all Borel sets. To this end, we will show that the right-hand side of (\ref{EqIntegralMeasure}) actually defines a Borel measure. That is, if we define
\begin{equation*}
    \mu(E) = \int_{a}^{b}\int_{H^{-1}(\{t\})\cap E} \frac{1}{|\nabla H|}\;d\mathscr{H}_{n-1\leq n}\;dt,
\end{equation*}
we will show that $\mu$ is a Borel measure. The previous lemma then implies that the open sets are in the domain of $\mu$ and that $\lambda^{n}(A) = \mu(A)$ for open sets (note that both measures are $\sigma$-finite since they coincide in open sets). Since two Borel measures that coincide on open sets are actually the same measure, it will follow that $\lambda^{n} = \mu$ and the Coarea Formula will follow.

The previous proofs will be carried out by the usual Dynkin System techniques, so we recall the pertinent concepts and results. We say that a family of sets $\Sigma$ is a \textbf{Dynkin system} if $\Omega\in\Sigma$, $\Sigma$ is closed under increasing countable unions and if $A,B\in\Sigma$ and $A\subset B$ then $B\setminus A\in\Sigma$. The intersection of any number of Dynkin systems is again a Dynkin system, hence given any family of subsets $\mathcal{A}$ there exists a minimal Dynkin system containing $\mathcal{A}$ denoted by $\mathcal{D}(\mathcal{A})$. The \textbf{Dynkin System Theorem} guarantees that if $\mathcal{A}$ is closed under finite intersections then $\mathcal{D}(\mathcal{A}) = \sigma(\mathcal{A})$. We now state the Coarea Formula and prove it.

\begin{theorem}[Coarea Formula. Particular Case.]
Let $H\colon \mathbb{R}^{n} \to \mathbb{R}$ be a smooth function. For any integrable function $f$ the equation
\begin{equation*}
    \int_{H^{-1}([a,b])} f\;d\lambda^{n} = \int_{a}^{b}\int_{H^{-1}(\{t\})} \frac{f}{|\nabla H|}\;d\mathscr{H}_{n-1\leq n}\;dt
\end{equation*}
is valid.
\end{theorem}
\begin{proof}
By the previous discussion, it is enough to show that $\mu$ is a Borel measure. First, for any set $E$ denote by $F_{E}(t)$ the function
\begin{equation*}
    t \longmapsto \int_{H^{-1}(\{t\})\cap E} \frac{1}{|\nabla H|}\;d\mathscr{H}_{n-1\leq n}.
\end{equation*}
By definition of $\mu$, we have that $\mu(E)$ is defined whenever $F_{E}$ is measurable. We define
\begin{equation*}
    \Sigma = \{E\subset\Omega\;|\;F_{E} \text{ is measurable}\}.
\end{equation*}
By the previous lemma we have that $\tau \subset \Sigma$, where $\tau$ is the topology of our space. We claim that $\Sigma$ is a Dynkin system, which, if true, would imply that $\mathcal{D}(\tau)\subset\Sigma$. Since topologies are closed under finite intersections, we actually have that $Bo(\tau) = \mathcal{D}(\tau)$ by the Dynkin System Theorem, hence $Bo(\tau) \subset \Sigma$ and $\mu$ would be defined on all Borel sets.

It is clear that $\Omega\in\Sigma$. Let $(A_{m})_{m\in\mathbb{N}}$ be an increasing sequence in $\Sigma$, then $F_{A_{m}}$ is measurable for any $m\in\mathbb{N}$. By the Monotone Convergence Theorem we have that
\begin{equation*}
    F_{\bigcup\limits_{m\in\mathbb{N}}A_{m}} = \lim_{m\to\infty}F_{A_{m}},
\end{equation*}
where the limit is pointwise. This implies that $F_{\bigcup\limits_{m\in\mathbb{N}}A_{m}}$ is measurable and $\bigcup\limits_{m\in\mathbb{N}}A_{m}\in\Sigma$. Consider $A,B\in\Sigma$ such that $A\subset B$. We have that
\begin{equation*}
    F_{B\setminus A} = F_{B} - F_{A}
\end{equation*}
and is thus measurable. This implies that $B\setminus A\in\Sigma$ and shows that $\Sigma$ is a Dynkin system. The previous considerations shows that $\mu$ is defined on $Bo(\tau)$.

The proof that $\mu$ is a measure is actually rather easy, as the only non-trivial part is showing $\sigma$-additivity. This is done by using the Monotone Convergence Theorem twice, once for each integral, the first time on increasing sets and the second one on a series of non-negative functions. It follows that $\mu$ is a Borel measure that coincides with $\lambda^{n}$ in open sets and hence $\lambda^{n} = \mu$. Finally, one approximates any integrable function with an increasing sequence of simple functions and uses the Monotone Convergence Theorem to conclude.
\end{proof}

This last theorem can be strengthened even more using the Rademacher Theorem to replace smoothness on $H$ by Lipschitz continuity, but we will not do so.

\subsection{Comparison with Hausdorff measure}\label{SubsecComparación}

We have seen that our measure $\mathcal{H}_{m\leq n}$ satisfies the Area and Coarea Formulas, bearing similitude with classical Hausdorff measure. In \cite{Federer} many other measures are constructed that satisfy the Area Formula, all of them defined in much larger $\sigma$-algebras than the Borel $\sigma$-algebra. We now study the relation between $\mathcal{H}_{m\leq n}$ and the classical Hausdorff measure, which we will denote as $H_{m\leq n}$. First note that due to the Area Formula we do now that $\mathcal{H}_{m\leq n}$ and $H_{m\leq n}$ coincide on all Borel subsets contained on an $m$-dimensional submanifold. To obtain finer comparisons we need to find better formulas for $\mathcal{H}_{m\leq n}$.

By definition, for any Borel set $A$ we have that
\begin{equation*}
    \mathcal{H}_{m\leq n}(A) = \lim_{\tilde{M}}\mu_{\tilde{M}}(A).
\end{equation*}
Since the net is increasing, we can rewrite this as
\begin{equation*}
    \mathcal{H}_{m\leq n}(A) = \sup_{\tilde{M}}\mu_{\tilde{M}}(A).
\end{equation*}
Recalling the definition of $\tilde{M}$ we find that
\begin{equation*}
    \mathcal{H}_{m\leq n}(A) = \sup\left\{\sum_{i = 1}^{k}\mu_{M_{i}}(A)\;\bigg|\;M_{i}\text{ disjoint $m$-dimensional submanifolds}\right\}.
\end{equation*}
Recalling the concept of summable net, we can also write this as
\begin{equation*}
    \mathcal{H}_{m\leq n}(A) = \sup\left\{\sum_{i \in I}\mu_{M_{i}}(A)\;\bigg|\;M_{i}\text{ disjoint $m$-dimensional submanifolds}\right\},
\end{equation*}
where $I$ is an arbitrary index set. Finally, we can also chose the $M_{i}$ to be compact submanifolds with boundary that intersect at most at the boundary, whether the sum is finite or not.

\begin{proposition}
$\mathcal{H}_{m\leq n} \leq H_{m\leq n}\restriction_{Bo(\mathbb{R}^{n})}$.
\end{proposition}
\begin{proof}
Let $A$ be a Borel set. If $H_{m\leq n}(A) = \infty$ then the inequality is immediate, hence we may assume that $H_{m\leq n}(A) < \infty$. We start from the formula
\begin{equation*}
    \mathcal{H}_{m\leq n}(A) = \sup\left\{\sum_{i = 1}^{k}\mu_{M_{i}}(A)\;\bigg|\;M_{i}\text{ disjoint $m$-dimensional submanifolds}\right\},
\end{equation*}
in which we need only consider compact submanifolds with boundary. Inner regularity of $H_{n\leq m}$ implies that
\begin{equation*}
    H_{m\leq n}(A) = \sup\{H_{m\leq n}(K)\;|\;K\subset A, \;K\text{ compact}\}.
\end{equation*}
Since $H_{m\leq n}(A\cap M) = \mu_{M}(A) = \mathcal{H}_{m\leq n}(A\cap M)$ we have that
\begin{equation*}
    \mathcal{H}_{m\leq n}(A) \leq H_{m\leq n}(A).
\end{equation*}
\end{proof}

The previous result uses only two fundamental facts about the Hausdorff measure:
\begin{enumerate}
    \item $H_{m\leq n}$ satisfies the Area Formula.
    \item $H_{m\leq n}$ is inner regular in sets of finite measure.
\end{enumerate}
Hence, we actually proved a minimality property for $\mathcal{H}_{m\leq n}$.

\begin{corollary}
The inductive volume measure $\mathcal{H}_{m\leq n}$ is the smallest Borel measure that satisfies the Area Formula and is inner regular in finite measure sets.
\end{corollary}

It follows from the previous results that $H_{m\leq n}$ null sets are also null for $\mathcal{H}_{m\leq n}$. From this we can finally obtain that $\mathcal{H}_{m\leq n}$ and $H_{m\leq n}$ coincide on a sufficiently large class of sets.

\begin{theorem}
Let $A$ be a Borel set. In the following cases the measures $\mathcal{H}_{m\leq n}$ and $H_{m\leq n}$ coincide:
\begin{enumerate}
    \item $A$ is contained in an $m$-dimensional subamifold (or a countable collection thereof).
    \item $A$ is a rectifiable set.
\end{enumerate}
\end{theorem}

The fact that both measures agree on rectifiable sets shows that for many analytical purposes the measure $\mathcal{H}_{m\leq n}$ is strong enough to be used in place of $H_{m\leq n}$. As previously stated, this is not enough to study irregular sets and dimension, however, this limitation was already present from the fact the domain of $\mathcal{H}_{m\leq n}$ is only the Borel sets. Hence, the fact that $\mathcal{H}_{m\leq n}$ and $H_{m\leq n}$ agree on rectifiable sets is more than enough if the intended use is the study of rectifiable sets or functions on rectifiable sets. Still, we develop further conditions for both measures to be equal based on covering results. 

Recall that if $A$ is a set and $\mathcal{V}$ is a cover of $A$, then $\mathcal{V}$ is a \textbf{Vitali class} if for each $x\in A$ and $\delta > 0$ there exists $U\in\mathcal{V}$ such that $x\in U$ and $0 < diam\;U < \delta$. The classical Vitali Covering Theorem for Hausdorff measure is the following (see \cite{Falconer}):

\begin{theorem}
Let $A$ be a $H_{m\leq n}$ measurable subset of $\mathbb{R}^{n}$ and $\mathcal{V}$ a Vitali class of closed sets that covers $A$. There exists a disjoint sequence $(U_{i})_{i\in\mathbb{N}}$ in $\mathcal{V}$ such that either
\begin{equation*}
    H_{m\leq n}\left(A\setminus\bigcup_{i\in\mathbb{N}}U_{i}\right) = 0
\end{equation*}
or
\begin{equation*}
    \sum_{i\in\mathbb{N}}(diam\;U_{i})^{m} = \infty.
\end{equation*}
\end{theorem}

In order to employ the previous results one usually needs to show that the second condition does not happen, hence the first condition must be true and a similar conclusion to the classical Vitali Covering Theorem can be achieved. Many general covering results can be found in \cite{Guzman1} and \cite{Guzman2}. With this we can prove the following.

\begin{theorem}
Let $A$ be a Borel set. If for every disjoint sequence of manifolds $(M_{k})_{k\in\mathbb{N}}$ the series
\begin{equation*}
    \sum_{k\in\mathbb{N}}diam(M_{k}\cap A)^{m}
\end{equation*}
is finite then $\mathcal{H}_{m\leq n}(A) = H_{m\leq n}(A)$.
\end{theorem}
\begin{proof}
Let $\mathcal{V}$ be the class of intersections of submanifolds of finite diameter with $A$. Then $\mathcal{V}$ is a covering of $A$ such that each element in $\mathcal{V}$ is a subset of $A$. Since we are allowing the elements of $\mathcal{V}$ to have arbitrarily small diameter, it is also a Vitali class, (except for isolated points, which may be removed without altering the value of either measure). The previous theorem implies that there is a disjoint sequence $\{C_{i}\}_{i\in\mathbb{N}}$ in $\mathcal{V}$ such that
\begin{equation*}
    H_{m\leq n}\left(A\setminus\bigcup_{i\in\mathbb{N}}U_{i}\right) = 0.
\end{equation*}
From the facts that $H_{m\leq n}(A) < \infty$, the elements of the sequence are disjoint and each $C_{i}\subset A$, we actually have that
\begin{equation*}
    H_{m\leq n}(A) = \sum_{i\in\mathbb{N}}H_{m\leq n}\left(C_{i}\right).
\end{equation*}
Using $\sigma$-additivity, the fact that $H_{m\leq n}(C_{i}) = \mathcal{H}_{m\leq n}(C_{i})$ and that $C_{i}\subset A$ we find that
\begin{align*}
    H_{m\leq n}(A) &= \sum_{i\in\mathbb{N}}H_{m\leq n}(C_{i})\\
    &= \sum_{i\in\mathbb{N}}\mathcal{H}_{m\leq n}(C_{i})\\
    &= \mathcal{H}_{m\leq n}\left(\bigcup_{i\in\mathbb{N}}C_{i}\right)\\
    &\leq \mathcal{H}_{m\leq n}(A).
\end{align*}
Since the opposite inequality is always verified, we conclude equality.
\end{proof}

\subsection{Riemannian Setting}\label{SecRiemannian}

The previous results carry on to the case where the euclidean space $\mathbb{R}^{n}$ is replaced by a riemannian manifold $S$. This is due to the geometric nature of our construction. The only noteworthy difference is in the proof of lemma \ref{LemaIntersección}, where the induced metric on the submanifold $N$ and $M$ is no longer given by
\begin{equation*}
    g^{N} = D\varphi_{M}^{t} D\varphi_{M}
\end{equation*}
but by
\begin{equation*}
    g^{N} = D\varphi_{M}^{t}\; g^{S}\; D\varphi_{M}.
\end{equation*}
The rest of the proof follows in the same way. All other proofs are the same as well. Hence, the inductive volume measure can be defined in riemannian manifolds and all previouss results remain valid.

\subsection{Concluding Remarks}

The inductive volume measure constructed in this paper is insufficient for studying irregular sets and Hausdorff dimension, but has many advantages including its simplicity when compared to Hausdorff measure and its minimality property. The same construction can be done on a riemannianan manifold without major differences and the results remain the same. This provides an alternative to Hausdorff measures in manifolds, which are much more difficult than that for euclidean spaces. Although we only proved the Area and Coarea Formulas, due to their complexity, many other results form Geometric Measure Theory can be proved much more easily, such as the Poincaré and Sobolev inequalities, the Gauss-Green Theorem, and Isoperimetric Inequalities. Other results related to Sobolev spaces such as the existence of trace and extension operators can be proved as well. In any case, the techniques of this paper provide a solid framework to develop further results related to geometric integration and serve as a solid introduction to these kinds of topics.

\section*{Acknowledgements}

\noindent \textbf{Data Availability} 
Data sharing is not applicable to this article as no new data were created or analyzed in this study.
\\

\noindent
\textbf{Declarations}\\

\noindent
\textbf{Conflict of interest} The authors declare that they have no conflict of interest.

\nocite{*}
\bibliographystyle{apsrev4-2}
\bibliography{main}
\addcontentsline{toc}{section}{Bibliography}

\end{document}